\documentclass[a4paper]{article}

\usepackage{amsfonts}
\usepackage{amsopn}
\usepackage{amsmath}
\usepackage{amsthm}
\usepackage{latexsym}
\usepackage{amssymb}

\newtheorem{theorem}{Theorem}[section]
\newtheorem{lemma}[theorem]{Lemma}
\newtheorem{remark}[theorem]{Remark}

\newtheorem{proposition}[theorem]{Proposition}
\newtheorem{example}[theorem]{Example}

\title{Anomalous solutions to nonlinear hyperbolic equations}

\author{Michael Oberguggenberger\thanks{Unit of Engineering Mathematics, University of Innsbruck,
Technikerstra\ss e 13, 6020 Innsbruck,
Austria, (michael.oberguggenberger@uibk.ac.at)}
}

\newcommand{\p}{\partial}
\newcommand{\eps}{\varepsilon}

\newcommand{\dd}{\hspace{1pt}{\rm d}\hspace{0.5pt}}

\newcommand{\ee}{{\rm e}\hspace{1pt}}

\newcommand{\ii}{{\rm i}\hspace{1pt}}

\newcommand{\R}{\mathbb R}
\newcommand{\N}{\mathbb N}
\newcommand{\C}{\mathbb C}

\newcommand{\cC}{{\mathcal C}}
\newcommand{\cD}{{\mathcal D}}
\newcommand{\cE}{{\mathcal E}}
\newcommand{\cF}{{\mathcal F}}
\newcommand{\cG}{{\mathcal G}}
\newcommand{\cH}{{\mathcal H}}

\newcommand{\cS}{{\mathcal S}}

\newcommand{\Cinf}{\ensuremath{{\mathcal C}^\infty}}

\newcommand{\WF}{\mathrm{WF}}
\newcommand{\singsupp}{\mathop{\mathrm{sing supp}}}
\newcommand{\supp}{\mathop{\mathrm{supp}}}
\renewcommand{\Re}{\mathop{\mathrm{Re}}}
\renewcommand{\Im}{\mathop{\mathrm{Im}}}
\newcommand{\arsinh}{\mathop{\mathrm{arsinh}}}

\newcommand{\hpl}{\ensuremath{\R\times [0,\infty) }}

\date{}

\begin{document}

\maketitle

\abstract{
The behavior of sufficiently regular solutions to semilinear hyperbolic equations has attracted a great deal of attention in the past decades, concerning local/global existence, finite time blow-up, critical exponents, and propagation of singularities. Solutions of lower regularity may exhibit unexpected (anomalous) propagation of singularities. The purpose of this paper is to present various striking examples that seemingly have not been addressed in the literature so far. The key issue is the interpretation of the nonlinear operations.
}

\section{Introduction}
\label{sec:intro}
This paper serves to display various unusual, or \emph{anomalous} solutions to semilinear wave equations
\begin{equation}\label{eq:NLWintro}
  \frac1{c^2}\p_t^2 u - \Delta u = f(x,t,u),\quad u(x,0) = u_0(x), \ \p_t u(x,0) = u_1(x)
\end{equation}
in space dimension $n\geq 1$, and to advection-reaction equations
\begin{equation}\label{eq:NLtransintro}
  \frac1{c}\p_t u +\p_x u = f(x,t,u),\quad u(x,0) = u_0(x)
\end{equation}
in one space dimension as prototypical hyperbolic partial differential equations.
For nonlinearities of the form $f(x,t,u) = \pm|u|^p$ or $\pm|u|^{p-1}u$, the main research direction in the past decades has been to find bounds on the exponent $p$ and the regularity of the initial data, asking about the existence of global solutions with small or large initial data, local solutions, self-similar solutions, blow-up in finite time or stability of blow-up. The reader is referred to the discussion in the monograph \cite{Reissig:2018}, the survey article from the 1990s \cite{Struwe:1992}, a collection of currently known critical exponents \cite{Lucente:2018} and some of the papers discussing the development of the field \cite{Georgiev:1997,Takamura:2016}. Relevant literature on self-similar solutions and stationary solutions as building blocks will be quoted at the appropriate place in Section \ref{sec:typeII}.

In order not to introduce additional singularities, the nonlinear function $f$ will be assumed to be smooth here (actually of the form $f(x,t,u) = g(x)u^p$ with integer $p \geq 2$).

In the 1980s and 1990s, a central question has been propagation of singularities, which started with the discovery of Jeffrey Rauch and Michael Reed \cite{RauchReed:1980,RauchReed:1981} that in semilinear hyperbolic equations and systems, singularities do not only propagate out from initial singularities along characteristics or bicharacteristics as in the linear case, but may be created at later times by the interaction of previous singularity bearing (bi-)characteristics. For example, an initial singularity at the origin in problem (\ref{eq:NLWintro}) may lead to singularities in the solution that fill up the solid light cone \cite{Beals:1983}. For a survey of the vast literature up to around 1990 we refer to the monograph \cite{Beals:1989}. Rauch and Reed coined the term \emph{anomalous singularities} for this phenomenon.

The results on anomalous singularities required sufficient overall regularity of the solution, for example $H^s_{\rm loc}$-regularity with $s > (n+1)/2$, and the mechanism for creating the anomalous singularities was still based on characteristics, bicharacteristics and their interaction.

The anomalous solutions presented in this paper are distinguished by
(a) lower regularity than in the previous literature and (b) propagation along non-characteristic curves.
The majority of examples is based on non-regular solutions to the corresponding stationary elliptic equation. Derivatives are always understood in the sense of distributions. In an attempt to categorize the solutions, four types will be singled out:
\smallskip

\begin{tabular}{ll}
\emph{Type I:} &products defined by H\"{o}rmander's wave front set criterion;\\[2pt]
\emph{Type II:} &products and powers evaluated by Nemytskii operators;\\[2pt]
\emph{Type III:} &limits of weak asymptotic solutions;\\[2pt]
\emph{Type IV:} &sequential solutions, especially very weak solutions in the\\[2pt]
     &sense of Ruzhansky.
\end{tabular}
\smallskip

\noindent
It is worth noting that all constructed solutions come with a certain assertion of uniqueness.

The plan of the paper is as follows. Section \ref{sec:regprop} serves to recall results on anomalous propagation of singularities for sufficiently regular solutions, for reasons of comparison. Section \ref{sec:typeI} addresses Type I solutions, introducing the employed multiplication of distributions and discussing the question of regularization. Section \ref{sec:typeII} will exhibit seemingly harmless solutions lying in an $L^p$-space on which the nonlinear operations are defined and continuous (Type II). In Section \ref{sec:typeIII} it will be shown that the solutions from Section \ref{sec:typeII} arise as limits of nets of asymptotic solutions (satisfying the equations up to an error term converging weakly to zero, Type III). In Section \ref{sec:typeIV} nets of smooth functions $(u_\eps)_{\eps > 0}$ will be constructed that solve the equations at each fixed $\eps > 0$, but need not necessarily converge as $\eps\to 0$ (Type IV). Nevertheless, their regularity properties can be characterized by suitable estimates on their growth in terms of negative powers of $\eps$ as $\eps\to 0$. The appendix serves to recall some notions required to define the products arising in Type I solutions.

The author has been aware of the existence of these anomalous solutions since the early 1980s, but due to a lack of explanation, hesitated to publish them so far. It is hoped that this publication will arouse interest in these types of solutions among the community. Many more examples of similar nature are known, collected by the author and in joint work with Hideo Deguchi \cite{Deguchi}.

What concerns notation, $H^s$ denotes the usual Sobolev space based on $L^2$; $\cC^k$ denotes the space of $k$-times differentiable functions, $\cC^k_b$ the subspace of functions with bounded derivatives up to order $k$. The notation for spaces of test functions and distributions follows \cite{Schwartz:1966}. The Fourier transform is defined as
$\cF\varphi(\xi) = \int \ee^{-2\pi \ii x\xi}\varphi(x)\dd x$.

\section{Propagation of singularities for regular solutions}
\label{sec:regprop}

This section serves to recall results from the 1980s on propagation of singularities for solutions to semilinear hyperbolic systems.
These results hold for sufficiently regular solutions ($L^\infty_{\rm loc}$ in one space dimension, $H^s_{\rm loc}$ for $s > (n+1)/2$ in
space dimension $n$). We do not strive for full generality -- the quoted results will be contrasted with the much less regular solutions to be constructed in the following sections.

We start with $(m\times m)$-systems of first order hyperbolic equations in one space dimension, considering the initial value problem
\begin{equation} \label{eq:hypsys}
   \begin{array}{l}
   (\p_t + \Lambda \p_x)u(x,t) = f(x,t,u(x,t)),\ (x,t) \in R\\[4pt]
   u(x,0) = u_0(x),\ x\in R_0
   \end{array}
\end{equation}
where $R_0 \subset\R$ is an interval and $R \subset \hpl$ is its domain of determinacy. Here $u = (u_1,\ldots,u_m)$,
$\Lambda = {\rm diag}(\lambda_1, \ldots, \lambda_m)$ with real and constant entries $\lambda_i$, and $f = (f_1,\ldots,f_m)$ is smooth.
Let $x_1, \ldots, x_k \in R_0$ and denote by  $S_0$ the union of characteristic lines emanating from $x_1, \ldots, x_k$.
Following \cite{RauchReed:1981}, construct the forward characteristic lines starting at the intersection points of $S_0$ and call this set $S_1$.
Let $S_2$ be the set of forward characteristic lines starting from the intersection points of $S_1$. Continue recursively to construct a sequence of sets $S_j$.
Let $S$ be the closure of $\bigcup_{j=0}^\infty S_j$ intersected with $R$.

\begin{proposition}
\label{prop:RR1981}
Let $u\in (L^\infty(R))^m$ satisfy (\ref{eq:hypsys}) in the sense of distributions and take on the initial data $u_0 \in (L^\infty(R_0))^m$.
Suppose that $u_0$ is $\Cinf$ with each derivative uniformly bounded on the complement of the finitely many points $x_1,\ldots, x_k$.
Then $u$ is $\Cinf$ on $R\setminus S$ and all derivatives of $u$ have continuous extensions from
each connected component of $R\setminus S$ to its closure.
\end{proposition}
\begin{proof}
This is Theorem 1 from \cite{RauchReed:1981}.
\end{proof}
\begin{remark}\label{rem:RR1981}
(a) If the function $f$ is linear, then the solution $u$ is in $\Cinf$ on $R\setminus S_0$ -- singularities can only lie on characteristic curves tracing back to the singularities of the initial data. In the nonlinear case, the solution is not $\Cinf$ on $S\setminus S_0$, in general. The singularities belonging to $S\setminus S_0$ in the nonlinear case have been termed \emph{anomalous singularities} by the authors.

(b) In the scalar case and in the case of $(2\times 2)$-systems (thus $m=1$ or $m=2$), $S = S_0$, so no anomalous singularities arise.
\end{remark}

Next we recall a result of \cite{Rauch:1979} on propagation of singularities for semilinear wave equations. Consider the initial value problem
\begin{equation} \label{eq:wave}
   \begin{array}{l}
   (\p^2_t - \Delta)v(x,t) = f(v(x,t)),\ (x,t) \in \R^n\times \R,\\[4pt]
   v(x,0) = v_0(x),\ \p_t v(x,0) = v_1(x),\ x\in \R^n,
   \end{array}
\end{equation}
where $f$ is a polynomial with $f(0) = 0$, $\Delta$ denotes the $n$-dimensional Laplace operator, and $u_0\in H^s_{\rm loc}(\R^n)$, $u_1\in H^{s-1}_{\rm loc}(\R^n)$
with $s > (n+1)/2$. Note that $H^s_{\rm loc}(\R^n\times\R)$ is an algebra in this case, even contained in the space of continuous functions, so $f(u)$ is classically defined.

\begin{proposition}
\label{prop:R1979}
Let $s>(n+1)/2$ and  $v\in H^s_{\rm loc}(\R^n\times\R)$ satisfy (\ref{eq:wave}) in the sense of distributions. Suppose that $v_0$ and $v_1$ belong to
$\Cinf(\R^n\setminus\{0\})$. Then $v$ is $\Cinf$ on $\{|x| > |t|\}$, and it belongs to $H^{s+1+\sigma}_{\rm loc}(\R^n\times\R)$ on $\{|x| < |t|\}$
for all $\sigma < s - (n+1)/2$.
\end{proposition}
\begin{proof}
This follows from Theorem 3.1, together with Theorem 1.1 of \cite{Rauch:1979}.
\end{proof}
\begin{remark}\label{rem:R1979}
In space dimension $n=1$, the solution $v$ is actually $\Cinf$ in $\{|x| < |t|\}$, as follows from the Corollary to Theorem 2 in \cite{RauchReed:1980} as well as the earlier paper \cite{Reed:1978}.
\end{remark}
It is known that the solution is not necessarily better than $H^{s+1+\sigma}$ in $\{|x| < |t|\}$ in space dimension $n \geq 2$. For a survey of the state of the art around 1990, see \cite{Beals:1989}.

\section{Type I solutions -- multiplication of distributions}
\label{sec:typeI}
In this section, we address weak solutions to nonlinear equations where the involved products or powers exist in the sense of H\"{o}rmander's wave front set criterion \cite{Hoermander:1971}. The examples will be based on the one-dimensional distribution
\begin{equation}\label{eq:u0}
   u_0(x) = \frac1{x+\ii 0} = \lim_{\eps\to 0}\frac1{x+\ii \eps} = {\rm vp}\frac1{x} - \ii\pi\delta(x)
\end{equation}
also denoted by $\delta_+(x)$ in the physics literature. Here ${\rm vp}\frac1{x}$ denotes the principal value distribution ${\rm vp}\frac1{x} = \p_x\log|x|$ and
$\delta(x)$ is the Dirac measure.
The Fourier transform of $u_0(x)$ and its auto-convolution are
\[
   (\cF u_0)(\xi) = -2\pi\ii H(\xi)\quad {\rm and}\quad (\cF u_0\ast\cF u_0)(\xi) = -4\pi^2 \xi H(\xi)
\]
where $H$ denotes the Heaviside function. In particular, the wavefront set of $u_0$ is $\{(0,\xi):\xi > 0\}$, thus $u_0^2$ exists according to H\"{o}rmander's criterion.
Actually, it can simply be computed as Fourier product (see Appendix),
\[
    u_0^2 = \cF^{-1}(\cF u_0\ast\cF u_0),
\]
as well as all its powers. It holds that
\begin{equation}\label{eq:u02}
   u_0^2(x) = \Big(\frac1{x+\ii 0}\Big)^2 = -\Big( \frac1{x+\ii 0}\Big)' = {\rm Pf}\frac1{x^2} + \ii\pi\delta'(x) = -u_0'(x)
\end{equation}
where ${\rm Pf}\frac1{x^2}$ is the Hadamard finite part distribution, and
\begin{equation}\label{eq:u03}
   2u_0^3(x) = 2\Big(\frac1{x+\ii 0}\Big)^3 = \Big( \frac1{x+\ii 0}\Big)'' = u_0''(x).
\end{equation}

\subsection{A nonlinear advection-reaction equation}
\label{subsec:ARE}
\begin{proposition}\label{prop:ARE}
The distribution $u(x,t) \equiv u_0(x)$ given by (\ref{eq:u0}) is a weak solution to the initial value problem
\begin{equation}\label{eq:ARp2}
  \frac1{c}\p_t u + \p_x u + u^2 = 0,\quad u(x,0) = u_0(x)
\end{equation}
for whatever $c\in\R, c\neq 0$, where the square is understood in the sense of H\"{o}rmander's product.
\end{proposition}
\begin{proof}
It is clear from (\ref{eq:u02}) that $\p_x u + u^2 = 0$ and that $\p_t u = 0$.
\end{proof}
Clearly, the mechanism producing this result is that the stationary solution satisfies the nonlinear differential relation $u_0' = - u_0^2$. Further reasons why a genuine distribution can satisfy such a relation will be discussed below. At first we wish to point out that the solution given in Proposition~\ref{prop:ARE} exhibits anomalous propagation of singularities. Indeed,
\[
  \singsupp u = \{(x,t): x = 0, t\geq 0\}
\]
while the expected singular support from Proposition~\ref{prop:RR1981} or Remark~\ref{rem:RR1981}(b) should be
$\{(x,t): x = ct, t\geq 0\}$. To be sure, $u_0$ does not belong to $L^\infty$ as required in Proposition~\ref{prop:RR1981}.

\begin{remark}\label{rem:otherspeeds}
It should be noted that anomalous propagation of singularities is not confined to stationary solutions. The following example, due to Deguchi \cite{Deguchi}, shows that any anomalous
propagation speed is possible. Indeed,
\begin{equation}\label{eq:otherspeeds}
   u(x,t) = \frac{1}{ax + bc t +\ii 0}
\end{equation}
with $a+b=1$ solves equation (\ref{eq:ARp2}) with initial data $u_0(x) = 1/(ax + \ii 0)$, noting that the Fourier product respects affine transformations of the independent variables. The singular support is
\[
  \singsupp u = \{(x,t): ax + bct = 0, t\geq 0\},
\]
which is a non-characteristic line if $a\neq b$.
\end{remark}

\begin{remark}
One possible explanation why the mentioned nonlinear differential relation, as well as similar relations for the higher derivatives, hold for the specific distribution (\ref{eq:u0}) can be obtained by studying its representation as a boundary value of an analytic function. Indeed, every distribution $v\in\cD'(\R)$ can be represented as the boundary value of a function $\widehat{v}(z)$, analytic in $\C\setminus\supp(u)$, in the sense
\begin{equation}\label{eq:BV}
v(x) =  \lim_{\eps\to 0}(\widehat{v}(x+\ii\eps) - \widehat{v}(x-\ii\eps))
\end{equation}
in $\cD'(\R)$, see e.g. \cite{Tillmann:1961}.
If $v$ is a distribution of compact support, $\widehat{v}(z)$ is given by the Fantappi\`{e} indicatrix
\[
   \widehat{v}(z) = \frac1{2\pi\ii}\big\langle v(x),\frac1{x-z}\big\rangle
\]
and in general by a partition of unity procedure.
Further, $|\widehat{v}(z)|$ grows at most like a negative power of $|\Im z|$ as $\Im z \to 0$, locally uniformly in $\Re z$. The representation $\widehat{v}(z)$
is unique up to a function analytic on $\C$. Further, every function $\widehat{v}(z)$, analytic in $\C\setminus\R$ and satisfying the growth condition has a distributional boundary value in the sense of (\ref{eq:BV}).

If the support of $\widehat{v}(z)$ is contained in $\{\Im z > 0\}$, the representation is unique.
Thus the space of distributions $\cH_+(\R)$ whose Fantappi\`e parametrix has support in the upper complex half plane is isomorphic to the space of analytic functions in the upper complex half plane satisfying the mentioned growth condition. However, the latter space is a differential algebra, the differential-algebraic structure of which can be transported to  $\cH_+(\R)$, rendering it a differential algebra \cite{Tillmann:1961}. (Similar constructions have also been elaborated in \cite{Ivanov:1979}.)

This is exactly the case with $u_0(x)$ given by (\ref{eq:u0}) for which
\[
   \widehat{u}_0(z) = \left\{ \begin{array}{ll} \frac1{z}, & \Im z > 0,\\[2pt]
                                0, & \Im z < 0.
                                \end{array}\right.
\]
In the algebra of analytic functions in the upper half plane, the functional relation
\[
    \frac{\dd^k}{\dd z^k}\Big(\frac1{z}\Big) = (-1)^k k! \Big(\frac1{z}\Big)^{k+1}, \quad z \neq 0
\]
is valid. In this way, formulas (\ref{eq:u02}) and (\ref{eq:u03}) are explained.
The differential-algebraic relations persist in the boundary values.
\end{remark}

\subsubsection{Analytic regularization}

It will be instructive to study the behavior of approximate solutions when the initial data are regularized. The first obvious possibility is to consider the analytic regularization defining the distribution $u_0(x) = 1/(x+\ii 0)$. We wish to solve the regularized problem
\begin{equation}\label{eq:AREreg}
  \frac{1}{c}\p_t u_\eps + \p_x u_\eps + u_\eps^2 = 0,\quad u_\eps(x,0) = u_{0\eps}(x) = \frac1{x +\ii \eps}.
\end{equation}
Solving (\ref{eq:AREreg}) by the method of characteristics results in the unique classical solution
\[
   u_\eps(x,t) = \frac{u_{0\eps}(x-ct)}{1 + ct u_{0\eps}(x-ct)} = \frac{\frac1{x-ct +\ii \eps}}{1 + ct \frac1{x-ct +\ii \eps}} = \frac1{x +\ii \eps}.
\]
Thus, by simple arithmetic, $ u_\eps(x,t) \equiv u_{0\eps}(x)$ and so the solution given in Proposition~\ref{prop:ARE} coincides with the weak limit of approximate solutions when the initial data are replaced by their analytic regularization.

\subsubsection{Regularization by convolution with a mollifier}

The purpose of this subsection is to show that the convergence of the approximate solution is a peculiarity of the analytic regularization and does not hold if the initial data are regularized by convolution with a standard Friedrichs mollifier
$\varphi_\eps(x) = \eps^{-1}\varphi(x/\eps)$ with $\varphi\in\cD(\R)$, $\int\varphi(x)\dd x = 1$.
For the sake of the argument, we take $\varphi\geq 0$ symmetric, $\supp\varphi \subset (-1,1)$. Thus let
\[
   U_{0\eps}(x) = (u_{0}\ast\varphi_\eps)(x)
\]
and let $U_\eps(x,t)$ be the corresponding classical solution to (\ref{eq:AREreg}) with initial condition $U_\eps(x,0) = U_{0\eps}(x)$. By the method of characteristics,
\[
   U_\eps(x,t) =  \frac{({\rm vp\frac{1}{x}}\ast \varphi_\eps)(x-ct) - \ii\pi \varphi_\eps(x-ct)}{1 + ct\big(({\rm vp\frac{1}{x}}\ast \varphi_\eps)(x-ct) - \ii\pi \varphi_\eps(x-ct)\big)}.
\]
In particular,
\begin{equation}\label{eq:AREblowup}
   U_\eps(ct-\eps,t) = \frac{({\rm vp\frac{1}{x}}\ast \varphi_\eps)(-\eps)}{1 + ct({\rm vp\frac{1}{x}}\ast \varphi_\eps)(-\eps)}.
\end{equation}
We show that the solution $U_\eps(x,t)$ blows up at latest at
\[
   t_\eps = \frac{- 1/c}{({\rm vp\frac{1}{x}}\ast \varphi_\eps)(-\eps)} = \frac{1/c}{({\rm vp\frac{1}{x}}\ast \varphi_\eps)(\eps)}
\]
and that this number is of order $\eps$ as $\eps\to 0$. Thus there is no global solution, when Friedrichs regularization is used.

Indeed, starting from the defining formula
\[
   ({\rm vp\frac{1}{x}}\ast \varphi_\eps)(x) = \lim_{\eta\to 0}\int_{|x-y|\geq\eta}\frac{\varphi_\eps(y)}{x-y}\dd y,
\]
some simple manipulations using the support properties of $\varphi$ lead to
\[
   ({\rm vp\frac{1}{x}}\ast \varphi_\eps)(-\eps) = \lim_{\eta\to 0}\int_{-1+\eta/\eps}^\infty \frac{\varphi(y)}{-\eps(1+y)}\dd y
       = -\frac1{\eps}\int_{\supp\varphi}\frac{\varphi(y)}{1+y}\dd y = -\frac1{\eps}C_\varphi
\]
where $C_\varphi$ is a positive constant. This shows that the denominator in (\ref{eq:AREblowup}) is indeed zero at $t_\eps = \eps/cC_\varphi$, while the numerator is nonzero.

\subsubsection{Separation in real and imaginary part}

One might argue that the complex valued initial value problem (\ref{eq:AREreg}) is actually a real valued, nonstrictly hyperbolic system. This is indeed the case; the real and imaginary part of the analytically regularized solution are
\[
  u_\eps(x,t) =\frac1{x+\ii \eps} = v_\eps(x,t) + \ii w_\eps(x,t) = \frac{x}{x^2+ \eps^2} - \ii \frac{\eps}{x^2+ \eps^2}.
\]
The hyperbolic system for the real and imaginary part is
\[
  \begin{array}{lcl}
   \p_t v_\eps + \p_x v_\eps &=& -v_\eps^2 + w_\eps^2,\\[4pt]
   \p_t w_\eps + \p_x w_\eps & =& -2v_\eps w_\eps.
   \end{array}
\]
Here $v_\eps(x,t)\to {\rm vp}\,\frac{1}{x}$ and  $w_\eps(x,t)\to -\pi \delta(x)$ as $\eps\to 0$. However, it is well-known (and rather immediate) that $v_\eps^2$ and $w_\eps^2$ do not converge in $\cD'(\R)$ as $\eps\to 0$. Thus the individual terms in the first line make no sense in the limit.
(By purely arithmetic manipulations involving $1/(x+\ii \eps)$ and $1/(x-\ii \eps)$ and their limits, the limit in the right-hand side of the second line is seen to exist and to equal $-\pi\delta'(x)$.)

\subsection{A nonlinear wave equation}
\label{subsec:NLW1D}

In the same vein, the distribution $u_0(x)$ can serve to produce a solution to a semilinear wave equation in one space dimension.
\begin{proposition}\label{prop:NLW1D}
The distribution $u(x,t) = u_0(x)$ given by (\ref{eq:u0}) is a weak solution to the initial value problem
\begin{equation}\label{eq:NLW1D}
  \frac1{c^2}\p_t^2 u - \p_x^2 u + 2u^3 = 0,\quad u(x,0) = u_0(x), \ \p_t u(x,0) = 0
\end{equation}
for whatever $c > 0$, where the cubic term is understood in the sense of H\"{o}rmander's product.
\end{proposition}
\begin{proof}
It is clear from (\ref{eq:u03}) that $-\p_x^2 u + 2u^3 = 0$ and that $\p_t u = 0$.
\end{proof}
In the real-valued case, the wave equation (\ref{eq:NLW1D}) has a so-called defocusing nonlinearity. For initial data $(u_0,u_1)$
in $H^1(\R)\times L^2(\R)$, it would have a unique global finite energy solution \cite{Struwe:1992}, belonging to
$\cC([0,\infty):H^1(\R))\cap\cC^1([0,\infty):L^2(\R))$. By local existence theory, it could also be extended to small negative times, and hence would belong to $L^\infty_{\rm loc}$ in an open neighborhood of the half plane. As in Remark~\ref{rem:R1979}, the Corollary to Theorem 2 in \cite{RauchReed:1980} would imply that a singularity in the initial data at $x=0$ can only spread along the characteristic lines $x=\pm ct$. Clearly, the solution given in Proposition~\ref{eq:NLW1D} neither has the required regularity properties nor does it show the expected singularity propagation.
\begin{remark}
(a) The distribution $u_0(x)$ is homogeneous of degree $-1$. Thus $u(x,t) = u_0(x)$ is a self-similar solution to (\ref{eq:NLW1D}), satisfying $\mu u(\mu x, \mu t) = u(x,t)$ for all $\mu > 0$.

(b) The function $u(x,t)$ from equation (\ref{eq:otherspeeds}) may serve as an example of a non-stationary solution to a nonlinear wave equation which exhibits anomalous propagation of singularities. Indeed,
when $a^2-b^2 = 1$, it solves equation (\ref{eq:NLW1D}) with initial data $u(x,0) = 1/(ax+\ii 0)$, $\p_t u(x,0) = 0$.
The initial singularity propagates along the line $\{(x,t): ax + bct = 0, t\geq 0\}$, which is non-characteristic if $a\neq \pm b$.
\end{remark}

\section{Type II solutions -- Nemytskii operators}
\label{sec:typeII}

This section addresses weak solutions, whereby the nonlinear terms are defined by Nemytskii operators.
We recall the \emph{pseudofunctions} $R_\lambda$, meromorphic functions of $\lambda\in\C$ with values in the space of tempered distributions $\cS'(\R^n)$ \cite[Chapter 17]{Dieudonne:1970}.
For ${\rm Re\,}\lambda > -n$ they are given by
\[
   \langle R_\lambda,\varphi\rangle = \int |x|^\lambda \varphi(x)\dd x
\]
and can be analytically continued to $\C \setminus\{-n-2k:k\in\N\}$. Outside the poles, they satisfy
\[
   \Delta R_\lambda = \lambda(\lambda + n -2)R_{\lambda - 2}.
\]
In particular, when $\lambda > 2-n$ and $p = 1 - 2/{\lambda}$, $R_\lambda$ belongs to $L^p_{\rm loc}(\R^n)$, $(R_\lambda)^p = R_{\lambda p}$ and it satisfies the elliptic equation
\[
   \Delta R_\lambda = \lambda(\lambda+n-2)(R_\lambda)^p,
\]
where the derivatives are understood in the weak sense and the $p$th power as the evaluation of the Nemytskii operator $L^p_{\rm loc}(\R^n) \to L^1_{\rm loc}(\R^n)$.

We note that for $\lambda \in \R \setminus\{-n-2k:k\in\N\}$, $R_\lambda$ is homogeneous of degree $\lambda$, and $R_\lambda \in H^1_{\rm loc}(\R^n)$, if $\lambda > (2-n)/2$.

As examples to be discussed further, we only consider two cases in which $p$ is a positive integer. In the context of propagation of singularities, fractional powers are not interesting for our purpose, because they represent non-smooth nonlinearities.
We use the solutions $R_\lambda$ as examples of peculiar rotationally symmetric stationary solutions to nonlinear wave equations.

\begin{example}\label{ex:p5}
Let $n = 3$ and $\lambda = -1/2$ (then $\lambda(\lambda + n -2) = -1/4$). Let $u_0(x) = |x|^{-1/2}$. Then $u_0\in L^5_{\rm loc}(\R^3)$, and $u(x,t) \equiv u_0(x)$ satisfies the nonlinear wave equation
\begin{equation}\label{eq:n3p5}
  \frac1{c^2}\p_t^2 u - \Delta u - \frac14 u^5 = 0,\quad u(x,0) = u_0(x), \ \p_t u(x,0) = 0
\end{equation}
for whatever $c>0$.
\end{example}
\begin{example}\label{ex:p3}
Let $n = 4$ and $\lambda = -1$ (then $\lambda(\lambda + n -2) = -1$). Let $u_0(x) = |x|^{-1}$. Then $u_0\in L^3_{\rm loc}(\R^4)$, and $u(x,t) \equiv u_0(x)$ satisfies the nonlinear wave equation
\begin{equation}\label{eq:n4p3}
  \frac1{c^2}\p_t^2 u - \Delta u - u^3 = 0,\quad u(x,0) = u_0(x), \ \p_t u(x,0) = 0
\end{equation}
for whatever $c>0$.
\end{example}
In all these cases, derivatives are understood in the weak sense and the powers of $u$ exist as locally integrable functions, actually as evaluations of the continuous map $u\to u^p$ from $L^p_{\rm loc}\to L^1_{\rm loc}$. Note that the nonlinear operation is taken outside the space of distributions, and the result is embedded afterwards.
\begin{remark}
(a) As $u_0$ is nonnegative, we might replace $u^5$ by $|u|^5$ or $|u|^4u$. In any case, we are dealing with so-called focusing nonlinearities.

(b) Recall that $u(x,t)$ is a self-similar solution to the nonlinear wave equation
\begin{equation}\label{eq:ss}
  \frac1{c^2}\p_t^2 u - \Delta u \pm|u|^p = 0,
\end{equation}
if $u(x,t) = \mu^\alpha u(\mu t, \mu x)$ for all $\mu >0$, where necessarily $\alpha = 2/(p-1)$. On the other hand, $u_0 = R_\lambda$ is homogeneous of degree $\lambda$, that is, $u_0(sx) = s^\lambda u_0(x)$ for $s>0$. It also satisfies equation (\ref{eq:ss}) when $\lambda - 2 = \lambda p$, i.e.,
$\lambda = - 2/(p-1)$. Thus the special solutions exhibited here are self-similar solutions to the nonlinear wave equation. However, they do not fall into the classes of functions considered e.g. in
\cite{Bizon:2007,Kato:2007,Pecher:2000a,Pecher:2000,Ribaud:2002}. It should be noted that solutions to nonlinear elliptic equations have also been used in the literature. They can serve for constructing solutions of finite life span, but also for proving the existence of (time-dependent) self-similar solutions \cite{Cote:2018,Donninger:2017,Kenig:2008,Krieger:2009}.
\end{remark}

\section{Type III -- weak asymptotic solutions}
\label{sec:typeIII}
A net of smooth functions $(u_\eps)_{\eps > 0}$ is a called a \emph{weak asymptotic solution} \cite{Danilov:2003} to a nonlinear partial differential equation, such as equation (\ref{eq:ss}), if it has a limit in the space of distributions and if it satisfies the equation
up to an error term which tends to zero weakly as $\eps\to 0$.

The basic example derives again from a nonlinear elliptic equation. Indeed, in $\R^n$, we start from the relation
\[
\Delta (|x|^2 + \eps^2)^{q} = \big(\big(2qn + 4q(q-1)\big)|x|^2 + 2qn\eps^2\big)(|x|^2 + \eps^2)^{q-2}.
\]
We will simply work out two special cases that correspond to the ones in Examples \ref{ex:p5} and \ref{ex:p3}.

\begin{example}\label{ex:p5e}
Let $n = 3$ and $q = -1/4$. By simple arithmetic,
\[
\big(2qn + 4q(q-1)\big)|x|^2 + 2qn\eps^2 = -\frac14(|x|^2 -\eps^2) - \frac54\eps^2
\]
and so
\[
\Delta (|x|^2 + \eps^2)^{-1/4}
    = -\frac14(|x|^2 + \eps^2)^{-5/4} - \frac54\eps^2(|x|^2 + \eps^2)^{-9/4}.
\]
Thus
\[
   u_{\eps}(x,t) = (|x|^2 + \eps^2)^{-1/4}
\]
satisfies the nonlinear wave equation
\begin{equation}\label{eq:n3p5e}
  \frac1{c^2}\p^2_t u_{\eps} - \Delta u_{\eps} -\frac14 u_{\eps}^5 - \frac54\eps^2u_{\eps}^{9} = 0
\end{equation}
for whatever $c>0$. An easy calculation shows that $\eps^2u_{\eps}^{9}$ converges to zero in $\cD'(\R^3)$ as $\eps \to 0$. Thus $u_{\eps}$ is a weak asymptotic solution to the nonlinear wave equation
(\ref{eq:n3p5}) with initial data converging to $u_0(x) = |x|^{-1/2}$. As in Example \ref{ex:p5} we set $u(x,t) = u_0(x)$. By the continuity assertions for Type II solutions,
\[
   u_\eps\to u,\quad u_\eps^5 \to u^5\quad{\rm in}\quad L_{\rm loc}^1(\R^3)\quad{\rm as\ }\eps\to 0,
\]
thus each term in equation (\ref{eq:n3p5e}) converges to the corresponding term in equation (\ref{eq:n3p5}). Further, $u_\eps$ is a smooth approximation to $u$; as $\eps\to 0$, a singularity emerges at $x=0$.
\end{example}
It is of interest to note that the solution to the regularized equation (\ref{eq:n3p5e}) is unique. This emphasizes again the anomaly in the propagation of singularities in the initial value problem (\ref{eq:n3p5}).
\begin{lemma}\label{lem:unique}
Let $n = 1$, $n=2$ or $n=3$. Assume that $u_0\in \cC^1_b(\R^n)$, $u_1\in \cC^0_b(\R^n)$ and let $f$ be smooth. Given any $T>0$, the initial value problem
\begin{equation}\label{eq:NLW}
  \frac1{c^2}\p_t^2 u - \Delta u = f(u),\quad u(x,0) = u_0(x), \ \p_t u(x,0) = u_1(x)
\end{equation}
has at most one weak solution in $\cC^0_b(\R^n\times[0,T])$.
\end{lemma}
\begin{proof}
Let $S(t)$ be the fundamental solution of the Cauchy problem, that is, $S(t)$ is the inverse Fourier transform of $\sin(c|\xi|t)/c|\xi|$. In space dimensions $n = 1,2,3$, $S(t)$ is a finite measure of total mass $ct$. The solution is given by
\[
   u(.,t) = \frac{d}{d t}S(t)\ast u_0 + S(t)\ast u_1 + \int_0^t S(t-s)\ast f(u(.,s))\dd s.
\]
By Young's inequality, the $L^\infty$-estimate
\[
   \|u(.,t)\|_{L^\infty(\R^n)} \leq C(t)\|u_0,\nabla u_0,u_1\|_{L^\infty(\R^n)}+\int_0^t (t-s)\|f(u(.,s))\|_{L^\infty(\R^n)}\dd s
\]
holds, where C(t) is a constant depending linearly on $t$. Applying this estimate to the difference $u-v$ of two solutions with the same initial data, writing $f(u) - f(v) = (u-v)g(u,v)$ with $g$ smooth and applying Gronwall's inequality shows that $u=v$.
\end{proof}

\begin{example}\label{ex:p3e}
Let $n = 4$ and $q = -1/2$ and let
\[
   u_{\eps}(x,t) = (|x|^2 + \eps^2)^{-1/2}.
\]
By the same arguments as in Example \ref{ex:p5e} one sees that $u_{\eps}$ satisfies the nonlinear wave equation
\begin{equation}\label{eq:n4p3e}
  \frac1{c^2}\p^2_t u_\eps - \Delta u_\eps - u_\eps^3 - 3\eps^2u_\eps^{5} = 0
\end{equation}
for whatever $c>0$. Again, one shows that $\eps^2u_{\eps}^{5}$ converges to zero in $\cD'(\R^4)$ as $\eps \to 0$, and $u_{\eps}$ is a weak asymptotic solution to the nonlinear wave equation (\ref{eq:n4p3}) with initial data converging to $u_0(x) = |x|^{-1}$. With $u(x,t) \equiv u_0(x)$, one has again
\[
   u_\eps\to u,\quad u_\eps^3 \to u^3\quad{\rm in}\quad L_{\rm loc}^1(\R^4)\quad{\rm as\ }\eps\to 0,
\]
thus each term in equation (\ref{eq:n4p3e}) converges to the corresponding term in equation (\ref{eq:n4p3}). The same behavior as in Example \ref{ex:p5e} is observed.
\end{example}
Due to the continuity of the Nemytskii operators, the weak asymptotic solutions constructed here are consistent with the solutions presented in Section \ref{sec:typeII}.

\section{Type IV -- sequential solutions}
\label{sec:typeIV}
In this section, we address solutions defined by nets of smooth functions which do not necessarily converge. To introduce the concept, let $\Omega$ be an open subset of $\R^n$ and let $P$ be a possibly nonlinear partial differential operator which is a smooth function of its arguments, $Pu = P(x,u,\p u,\ldots)$. Let $(u_\eps)_{\eps>0}$ a net of functions belonging to $\Cinf(\Omega)$. If $Pu_\eps = 0$ for all sufficiently small $\eps>0$, then the net $(u_\eps)_{\eps>0}$ is called a \emph{sequential solution} of the equation $Pu = 0$, following e.g. \cite{Rosinger:1980}. The net $(u_\eps)_{\eps>0}$ may or may not converge. Even if $(u_\eps)_{\eps>0}$ converges, individual terms in $P(x,u,\p u,\ldots)$ may or may not converge. However, if $(u_\eps)_{\eps>0}$ converges to a distribution $u$, together with all individual terms in $P(x,u,\p u,\ldots)$, then $u$ can be called a \emph{proper weak solution} to $Pu=0$ \cite{Lindblad:1996}.

Restricting the class of sequential solutions to \emph{moderate nets} allows one to establish a regularity theory for sequential solutions, even if they diverge.
A net of smooth functions $(u_\eps)_{\eps>0}$ on $\Omega$ is called \emph{moderate}, if for all compact subsets $K$ of $\Omega$ and all multi-idices
$\alpha \in \mathbb{N}_0^n$ there exists $b\geq 0$ such that
\[
	\textstyle\sup_{x \in K} |\partial^{\alpha} u_{\varepsilon}(x)| = O(\varepsilon^{-b}) \quad {\rm as}\ \varepsilon \to 0.
\]
The net of smooth functions $(u_\eps)_{\eps>0}$ on $\Omega$ is called \emph{negligible}, if for all compact subsets $K$ of $\Omega$, all multi-indices
$\alpha \in \mathbb{N}_0^n$ and all $a\geq 0$,
\[
	\textstyle\sup_{x \in K} |\partial^{\alpha} u_{\varepsilon}(x)| = O(\varepsilon^a) \quad {\rm as}\ \varepsilon \to 0.
\]
Following \cite{Garetto:2015,Ruzhansky:2017}, a moderate net satisfying $Pu_\eps = 0$ for all sufficiently small $\eps > 0$ is called a \emph{very weak solution} to the equation $Pu=0$.
If $(u_\eps)_{\eps>0}$ is moderate and $Pu_\eps = n_\eps$ where $(n_\eps)_{\eps>0}$ is a negligible net, then $(u_\eps)_{\eps>0}$ is a \emph{Colombeau solution} to the equation $Pu=0$. (As a matter of fact, its equivalence class in the Colombeau algebra $\cG(\Omega)$ is a solution in the differential-algebraic sense \cite{Colombeau:1985,Grosser:2001,MO:1992}.)

Finally, a net $(u_\eps)_{\eps>0}$ is said to possess the $\cG^\infty$-property, if for all compact subsets $K$ of $\Omega$ there is $b\geq 0$ such that for all multi-indices
$\alpha \in \mathbb{N}_0^n$,
\[
	\textstyle\sup_{x \in K} |\partial^{\alpha} u_{\varepsilon}(x)| = O(\varepsilon^{-b}) \quad {\rm as}\ \varepsilon \to 0.
\]
(Note the change in quantifiers: the local order of growth is the same for all derivatives.) The significance of this notion is that it generalizes $\Cinf$-smoothness from distributions to moderate nets. In fact, if $w\in\cE'(\Omega)$ is a compactly supported distribution and $\varphi_\eps$ is a mollifier ($\varphi_\eps(x) = \eps^{-n}\varphi(x/\eps)$ with $\varphi$ smooth, rapidly decaying and $\int\varphi(x)\dd x = 1$), then
\begin{itemize}
\item $w_\eps = w\ast\varphi_\eps|\Omega$ defines a moderate net;
\item $(w_\eps)_{\eps>0}$ has the $\cG^\infty$-property if and only if $w\in\cC^\infty(\Omega)$.
\end{itemize}
The $\cG^\infty$-singular support of a moderate net $(u_\eps)_{\eps>0}$ is defined as the complement of the largest open subset $\omega\subset\Omega$ such that
$(u_\eps|\omega)_{\eps>0}$ has the $\cG^\infty$-property on $\omega$. The same notions can be introduced for nets of smooth functions defined on the closure of an open subset of $\R^n$, thereby enabling the study of initial value problems or boundary value problems.

Replacing $\Cinf$ by $\cG^\infty$, classical regularity theory and propagation of singularities for linear partial differential equations can be literally transferred to the setting of moderate nets in the case of linear equations (with possibly non-smooth coefficients).
Here are some specific results in this direction: $\cG^\infty$-singularities in the linear wave equation propagate along the light cone in any space dimension, \cite{MO:1992}. For wave equations in one space dimension with piecewise constant coefficient, propagation of $\cG^\infty$-singularities occurs along characteristic lines emanating from the initial point singularity, with reflection/diffraction at the points of discontinuity of the coefficient, \cite{Deguchi:2016}. The $\cG^\infty$-wave front set of the kernels of Fourier integral operators can be computed analogously to the classical case, and $\cG^\infty$-singularities in solutions to first order hyperbolic equations propagate along the Hamiltonian flow \cite{Garetto:2014}.

\subsection{Moderate sequential solutions to an advection-reaction equation}
\label{sec:modtrans}
We are going to construct moderate sequential solutions to the advection-reaction equation in one space dimension
\begin{equation}\label{eq:modseqtrans}
  \frac{1}{c}\p_t u + \p_x u + \frac{2}{p}\,x\,u^{p+1} = 0,\quad u(x,0) = u_0(x)
\end{equation}
where -- for simplicity -- $p$ is a positive integer. We first note that for continuous initial data, there is at most one solution.
\begin{lemma}\label{lem:transunique}
Assume that $u_0\in \cC^0_b(\R)$, $c\neq 0$ and let $f$ be smooth. Given any $T>0$, the initial value problem
\begin{equation}\label{eq:NLtrans}
  \frac1{c}\p_t u +\p_x u = f(x,t,u),\quad u(x,0) = u_0(x)
\end{equation}
has at most one weak solution in $\cC^0_b(\R^n\times[0,T])$.
\end{lemma}
\begin{proof}
Indeed, if $u$ is a solution, it solves the integral equation
\[
   u(x,t) = u_0(x-ct) + \int_0^t f(x-ct+cs,s,u(x-ct+cs,s))\dd s.
\]
Uniqueness follows by the same argument as in the proof of Lemma \ref{lem:unique}.
\end{proof}
It is immediately checked that, for each $\eps >0$, the smooth function
\begin{equation}\label{eq:ueps}
   u_\eps(x,t) \equiv u_{0\eps}(x) = (x^2 + \eps^2)^{-1/p}
\end{equation}
is a solution to the initial value problem
\begin{equation}\label{eq:modseqtranse}
\frac1{c}\p_t u_\eps + \p_x u_\eps + \frac{2}{p}\,x\,u_\eps^{p+1} = 0, \quad u_\eps(x,0) = (x^2 + \eps^2)^{-1/p}.
\end{equation}
According to Lemma \ref{lem:transunique}, the solution is unique.
It is clear that the net $(u_\eps)_{\eps > 0}$ is moderate, hence it defines a moderate sequential solution to (\ref{eq:modseqtrans}).
\begin{lemma}\label{lemdivergence}
The net $(u_{0\eps})_{\eps > 0}$ converges for $p\geq 3$ and diverges for $p=1,2$. In particular, $(u_{0\eps}^{p+1})_{\eps > 0}$ diverges for every $p>0$.
\end{lemma}
\begin{proof}
For $p\geq 3$, $u_0(x) = |x|^{-2/p}$ belongs to the space of locally integrable functions, and $u_{0\eps}(x) = (x^2+\eps^2)^{-1/p}$ converges to it in that space.

Let $p=2$ and take a test function $\varphi \geq 0$ such that $\varphi(x) = 1$ on $[-1,1]$. Then
\begin{eqnarray*}
\langle u_{0\eps},\varphi\rangle = \int_{-\infty}^\infty\frac{\varphi(x)}{\sqrt{x^2+\eps^2}}\dd x
   \geq \int_{-1}^1\frac{1}{\sqrt{x^2+\eps^2}}\dd x = \int_{-1/\eps}^{1/\eps}\frac{1}{\sqrt{y^2+1}}\dd y \to\infty
\end{eqnarray*}
as $\eps\to 0$. A similar argument shows that $(x^2+\eps^2)^{-q}$ diverges for $q > 1/2$. Thus $u_{0\eps}(x) = (x^2+\eps^2)^{-1/p}$ diverges when $p<2$ as well, in particular, for $p=1$. Further, $u_{0\eps}^{p+1}(x) = (x^2+\eps^2)^{-1-1/p}$ diverges for every $p>0$.
\end{proof}
This shows that even in the convergent case $p\geq 2$, the limit $u = \lim_{\eps\to 0}u_\eps$ is not a proper solution of equation (\ref{eq:modseqtrans}).

\subsubsection{The special case $p=2$}

Let us have a more detailed look at the (divergent) case $p=2$. Then the function
\begin{equation}\label{eq:solutrans0.5}
     u_\eps(x,t) = (x^2 + \eps^2)^{-1/2},
\end{equation}
at fixed $\eps > 0$, is a solution to the advection-reaction equation
\begin{equation}\label{eq:modseqtrans0.5}
   \frac{1}{c}\p_t u_\eps + \p_x u_\eps  + x\,u_\eps^{3} = 0,\quad u_\eps(x,0) = (x^2 + \eps^2)^{-1/2}.
\end{equation}
According to Lemma \ref{lem:transunique}, this solution is unique. We may study its $\cG^\infty$-regularity properties.
\begin{proposition}\label{prop:transreg}
The $\cG^\infty$-singular support of $(u_\eps)_{\eps > 0}$ is $\{(0,t): t\geq 0\}$.
\end{proposition}
\begin{proof}
Let $\chi(x) = (x^2+1)^{-1/2}$. Then $u_\eps(x,t) = (x^2 + \eps^2)^{-1/2} = \chi_\eps(x) = \eps^{-1}\chi(x/\eps)$. It is straightforward to show that the $k$th derivative of $\chi$ is of the form
\[
   \chi^{(k)}(x) = P_k(x)(x^2+1)^{-k-1/2}
\]
where $P_k$ is a polynomial of degree $k$. Therefore,
\[
  \chi_\eps^{(k)}(x) = \eps^{-k-1}P_k\big(\frac{x}{\eps}\big)\big(\frac{x^2}{\eps^2}+1\big)^{-k-1/2} = \eps^kP_k\big(\frac{x}{\eps}\big)(x^2+\eps^2)^{-k-1/2}.
\]
When $|x| \geq x_0 > 0$, the latter expression is bounded independently of $\eps > 0$. Thus $(u_\eps)_{\eps > 0}$ has the $\cG^\infty$-property in the region
$\{(x,t):|x| > 0, t\geq 0\}$.

On the other hand, $\chi(x)$ is the derivative of $\arsinh x$, whose Taylor expansion shows that $\chi^{(k)}(x) \neq 0$ when $k$ is an even integer. Thus
\[
\chi_\eps^{(k)}(0) = \eps^{-k-1}\chi^{(k)}(0)
\]
does not have the $\cG^\infty$-property: the line $x=0$ is contained in the $\cG^\infty$-singular support.
\end{proof}
This shows that the moderate sequential solution to (\ref{eq:modseqtrans0.5}) exhibits anomalous propagation of singularities. The initial $\cG^\infty$-singularity at $x=0$ is not propagated along the line $x=ct$ as in the linear case, but rather remains at $x=0$ for all times.
\begin{remark}
Actually, the classical initial value problem $\frac{1}{c}\p_t v + \p_x v  + x\,v^{3} = 0$, $v(x,0) = v_0(x)$ can be solved explicitly. Transformation to characteristic coordinates $s = t$, $y = x-ct$ leads to an ordinary differential equation and to the solution
\[
 v(x,t) = \frac{v_0(x-ct)}{\sqrt{\big(x^2 - (x-ct)^2\big)v_0^2(x-ct) + 1}}.
\]
Inserting $v_0(x) = (x^2 + \eps^2)^{-1/2}$ it turns out that by simple arithmetic, $v(x,t) = (x^2 + \eps^2)^{-1/2}$, supporting the fact that $u_\eps(x,t)$ as given above by (\ref{eq:solutrans0.5}) is indeed the solution. The same phenomenon also happens for $p\neq2$
in (\ref{eq:ueps}) and (\ref{eq:modseqtranse}).
\end{remark}

\subsection{Moderate sequential solutions to a nonlinear wave equation}
\label{sec:modNLW}
Taking a further $x$-derivative, it is seen that $u_\eps(x,t)$ given by (\ref{eq:solutrans0.5}) also solves the one-dimensional nonlinear wave equation
\[
   \frac{1}{c^2}\p_t^2 u_\eps - \p_x^2 u_\eps + u_\eps^3 + 3x^2 u_\eps^5 = 0,\quad u_\eps(x,0) = (x^2 + \eps^2)^{-1/2},\ \p_tu_\eps(x,0) = 0
\]
for every $c>0$. In this case, standard energy estimates can be used to show that the solution is unique.
\begin{lemma}\label{lem:uniqueNLWseq}
Given $v_0\in H^1(\R)$, $v_1\in L^2(\R)$ of finite energy (defined by (\ref{eq:energy}) below), the equation
\begin{equation}\label{eq:NLWseq}
   \frac{1}{c^2}\p_t^2 v - \p_x^2 v + v^3 + 3x^2 v^5 = 0,\quad v(x,0) = v_0(x),\ \p_tv(x,0) = v_1(x)
\end{equation}
has a unique solution $v\in \cC([0,\infty):H^1(\R))\cap \cC^1([0,\infty):L^2(\R))$ of finite energy, where $c>0$.
\end{lemma}
\begin{proof}
It is quite obvious that the energy
\begin{equation}\label{eq:energy}
  E(t) = \frac12\int_{-\infty}^\infty\Big(|\p_t v|^2 + c^2|\p_x v|^2 + |v|^4 + 3x^2|v|^6\Big)\dd x
\end{equation}
is conserved. The proof follows standard arguments (see e.g. \cite{Struwe:1992}).
\end{proof}
At fixed $\eps > 0$, $u_\eps(x,0) = (x^2 + \eps^2)^{-1/2}$ belongs to $H^1(\R)$ and, together with $\p_tu_\eps(x,0) =0$, forms initial data of finite energy.
Thus the stationary solution $u_\eps(x,t) = u_\eps(x,0)$ is the unique solution in this sense. The net $(u_\eps)_{\eps > 0}$ provides a moderate sequential solution to the nonlinear wave equation (\ref{eq:NLWseq}). Its $\cG^\infty$-singular support $\{(x,t), x = 0, t \geq 0\}$ has been computed in Proposition \ref{prop:transreg}.
Again, this differs from the linear case \cite{Garetto:2014} and the nonlinear, classical case (Propositions \ref{prop:RR1981}, \ref{prop:R1979}), according to which the singular support should be $\{(x,t), |x| = ct, t \geq 0\}$.

Anomalous propagation of singularities persists for sequential solutions.
\bigskip

{\bf Acknowledgements:}
I wish to thank the organizers of the INdAM Workshop \emph{Anomalies in Partial Differential Equations} for providing an attractive environment for presenting the results of the paper. Discussions with various participants led to further insight. In particular, I would like to thank Lavi Karp, Sandra Lucente, Alberto Parmeggiani, Michael Reissig, Luigi Rodino and Michael Ruzhansky for helpful remarks. Thanks go to Hideo Deguchi for joint work on the topic since 2016.

\section*{Appendix: On multiplication of distributions}
\addcontentsline{toc}{section}{Appendix: Multiplication of distributions}
\label{Appendix}

Let $S,T\in \cS'(\R^n)$. The $\cS'$-convolution of $S$ and $T$ is said to exist, if
\[
   (\varphi\ast\check{S})T \in \cD_{L^1}'(\R^n),\quad {\rm for\ all}\quad \varphi\in \cS(\R^n),
\]
where $\check{S}(x) = S(-x)$. In this case, the convolution is defined by $\langle S\ast T,\varphi\rangle = \langle (\varphi\ast\check{S})T, 1 \rangle$, and $S\ast T$ belongs to $\cS'(\R^n)$.

Let $u,v\in \cS'(\R^n)$. If the $\cS'$-convolution of $\cF u$ and $\cF v$ exists, one may define the \emph{Fourier product}
\[
   u\cdot v = \cF^{-1}(\cF u \ast \cF v).
\]
The definition can be localized \cite{Ambrose:1980} as follows. Assume that for every $x\in\R^n$ there is a neighborhood $\Omega_x$ and $\chi_x\in\cD(\R^n)$, $\chi_x\equiv 1$ on $\Omega_x$, such that the $\cS'$-convolution of $\cF(\chi_x u)$ and $\cF(\chi_x v)$ exists. Locally near $x$, the product $u\cdot v$ is defined to be
$\cF^{-1}(\cF(\chi_x u) \ast \cF(\chi_x v))$. Globally, it is defined by a partition of unity argument.

A special case arises when the distributions satisfy H\"ormander's wave front set criterion \cite{Hoermander:1971}, requiring that for every
$(x,\xi)\in \R^n\times(\R^{n}\setminus\{0\})$, $(x,\xi)\in\WF(u)$ implies $(x,-\xi)\not\in\WF(v)$.

In space dimension $n=1$, a very convenient case arises when $\supp\cF u$ and $\supp\cF v$ are contained in $[0,\infty)$. (In particular, H\"{o}rmander's criterion is fulfilled.)
The basic example used in Section \ref{sec:typeI} is
\[
   u_0(x) = \frac1{x+\ii 0} = \lim_{\eps\to 0}\frac1{x+\ii \eps} = {\rm vp}\frac1{x} - \ii\pi\delta(x)
\]
whose Fourier transform is $(\cF u_0)(\xi) = -2\pi\ii H(\xi)$. The auto-convolution results in $(\cF u_0\ast\cF u_0)(\xi) = -4\pi^2 \xi H(\xi)$.
Thus $u_0^2 = \cF^{-1}(\cF u_0\ast\cF u_0)$ exists as Fourier product, and the formula shows that $u_0^2(x) = - u_0'(x)$. The remaining formulas used in Section \ref{sec:typeI} follow in the same way.

A more general definition of the product of distributions on $\R^n$ can be obtained by regularization and passage to the limit. The \emph{model product} of $u$ and $v$ is defined as
\[
   [u\cdot v] = \lim_{\eps\to 0}(u\ast\varphi_\eps)(v\ast\varphi_\eps)
\]
provided the limit exists for all mollifiers $\varphi_\eps$ of the form $\varphi_\eps(x) = \eps^{-n}\varphi(x/\eps)$ with $\varphi\in\cD(\R^n)$, $\int\varphi(x)\dd x = 1$, and is independent of the chosen mollifier. If the Fourier product exists, so does the model product.

In the one-dimensional case ($n=1$), a yet more general definition is obtained by using the representation by boundary values of analytic functions, which was discussed in Section \ref{sec:typeI}. Given $u\in\cD'(\R)$, let
\[
   \widetilde{u}_\eps(x) =  \widehat{u}(x+\ii\eps) - \widehat{u}(x-\ii\eps),
\]
with the right-hand side as in (\ref{eq:BV}). It was seen in Section \ref{sec:typeI} that $u(x) = \lim_{\eps\to 0}\widetilde{u}_\eps(x)$. If $u\in\cD'_{L^1}(\R)$, $\widetilde{u}_\eps$ is obtained by convolving $u$ with the special mollifier $\psi_\eps(x) = \eps/(\pi(x^2+\eps^2))$. The {\emph{Tillmann product} \cite{Tillmann:1961} of two distributions $u, v$ is defined by
\[
u\circ v = \lim_{\eps\to 0}\widetilde{u}_\eps\cdot\widetilde{v}_\eps
\]
provided the limit exists. The definition does not work in higher space dimensions; there, harmonic regularization should be used \cite{Boie:1989}. In any case, the powers in
(\ref{eq:u02}) and (\ref{eq:u03}) can also be understood in the sense of the Tillmann product.

H\"{o}rmander's criterion implies the existence of the Fourier product, which implies the existence of the model product and in turn also the existence of the Tillmann product. None of the implications can be reversed.

The other products used in this paper enter at different levels. For example, the most basic product of a smooth function with a distributions enters below H\"{o}rmander's criterion. The product in $H^s_{\rm loc}(\R^n)$ when this space is an algebra ($s > n/2$) enters as a subcase of the Fourier product, but is independent of H\"{o}rmander's criterion. The Nemytskii operators in the form of a continuous map $L^p_{\rm loc}\times L^q_{\rm loc}\to L^1_{\rm loc}$, $1/p + 1/q = 1$, enter at the level of the model product, but are independent of the Fourier product criterion. For more details on these circle of ideas, see \cite{MO:1992}.

\end{document}